\documentclass[11pt]{article}
\usepackage{amsmath,amssymb,amsfonts}
\usepackage{amsthm}
\usepackage{cite}
\newtheorem{theorem}{Theorem}
\newtheorem{definition}{Definition}
\newtheorem{lemma}{Lemma}

\title{\textbf{Large Deviation Principle for Some Measure-Valued Processes} }
\author{Parisa Fatheddin and Jie Xiong\thanks{Research supported partially by
NSF DMS-0906907.}\\
\emph{University of Tennessee}}
\date{}
\begin{document}
\newtheorem{example}[theorem]{Example}
\newtheorem{cor}[theorem]{Corollary}
\newtheorem{notation}[theorem]{Notation}
\newtheorem{notations}[theorem]{Notations}
\newtheorem{claim}[theorem]{Claim}
\newtheorem{mtheorem}[theorem]{Meta-Theorem}
\newtheorem{prop}[theorem]{Proposition}
\newtheorem{rem}[theorem]{Remark}
\newtheorem{conj}[theorem]{Conjecture}
\newtheorem{rems}[theorem]{Remarks}

\def\al{{\alpha}}\def\be{{\beta}}\def\de{{\delta}}
\def\ep{{\epsilon}}\def\ga{{\gamma}}\def\ka{{\kappa}}
\def\la{{\lambda}}\def\om{{\omega}}\def\si{{\sigma}}
\def\ze{{\zeta}}\def\De{{\Delta}}
\def\Ga{{\Gamma}}
\def\La{{\Lambda}}\def\Om{{\Omega}}
\def\th{{\theta}}\def\Th{{\Theta}}
\def\<{\left<}\def\>{\right>}\def\({\left(}\def\){\right)}

\newfam\msbmfam\font\tenmsbm=msbm10\textfont
\msbmfam=\tenmsbm\font\sevenmsbm=msbm7

\scriptfont\msbmfam=\sevenmsbm\def\bb#1{{\fam\msbmfam #1}}

\def\AA{\bb A}\def\BB{\bb B}\def\CC{\bb C}
\def\EE{\bb E}\def\GG{\bb G}\def\HH{\bb H}\def\KK{\bb K}
\def\LL{\bb L}\def\NN{\bb N}\def\PP{\bb P}\def\QQ{\bb Q}
\def\RR{\bb R}\def\TT{\bb T}\def\WW{\bb W}\def\ZZ{\bb Z}
\def\cA{{\cal A}}\def\cB{{\cal B}}\def\cL{{\cal L}}
\def\cD{{\cal D}}\def\cF{{\cal F}}\def\cG{{\cal G}}
\def\cH{{\cal H}}\def\cI{{\cal I}}\def\cJ{{\cal J}}
\def\cL{{\cal L}}\def\cM{{\cal M}}
\def\cP{{\cal P}}\def\cS{{\cal S}}\def\cT{{\cal T}}\def\cU{{\cal U}}\def\cV{{\cal V}}
\def\cW{{\cal W}}\def\cX{{\cal X}}\def\cY{{\cal Y}}\def\cZ{{\cal Z}}

\def\us{{\underline{s}}}
\def\ut{{\underline{t}}}
\def\uu{{\underline{u}}}

\def\goto{{\rightarrow}}

\maketitle

\begin{abstract}
 We establish a large deviation principle for the solutions of a class of
 stochastic partial differential equations with non-Lipschitz continuous
 coefficients. As an application, the large deviation principle is derived for
 super-Brownian motion and Fleming-Viot process.
 \end{abstract}



 \section{Introduction}

 Measure-Valued processes (MVP) arise from many fields of applications including population growth models
 and genetics. We refer the reader to the books of Dawson \cite{Daw},
 Etheridge \cite{Eth}, Perkins \cite{Per}, and Li \cite{Li} for an introduction to this topic.
 Two of the most studied measure-valued processes are super-Brownian motion (SBM) and
 Fleming-Viot process (FVP). An interesting problem concerns the limiting behavior of these processes
 when the branching rate (for SBM) or the mutation rate (for FVP) $\epsilon$, tends to zero. It is easy to see that the
 measure-valued processes, denoted by $\mu_{t}^{\epsilon}$, converge to a deterministic measure-valued process $\mu_{t}^{0}$,
 and it is desirable to study this rate of convergence.

 Large deviation principle (LDP) is a very useful tool for the study of convergence rate. Roughly speaking, the goal of the
 LDP is to determine the rate $R(\delta) >0$, for any $\delta >0 $ such that as $\epsilon \rightarrow 0$,
 \begin{equation}\label{1.1}
 P\left( \rho \left( \mu^{\epsilon}, \mu^{0}\right)> \delta \right) \approx \exp\left(-\epsilon^{-1} R(\delta)\right),
 \end{equation}
 for a suitable distance, $\rho$  in $C\left([0,1]; \mathcal{M}_{\beta}(\mathbb{R})\right)$,
 the state space of the MVP, where $\mathcal{M}_{\beta}(\mathbb{R})$ is the set of $\sigma$-finite measures $\mu$ on
$\mathbb{R}$ such that
\begin{equation}\label{e}
\int e^{-\beta|x|}d\mu(dx)< \infty.
\end{equation}
 We refer the reader to the books of Dembo and
 Zeitouni \cite{Dem}, Deuschel and Stroock \cite{Deu}, and Dupuis
 and Ellis \cite{Dup} for more background on this subject.

 LDP for MVP has been studied by many authors. Fleischmann and Kaj \cite{FK} proved the LDP for SBM for a fixed time
 $t$. Later on, sample path LDP for SBM was derived independently by Fleischmann $et\;al$ \cite{FGK},
  and Schied \cite{Sch} while the rate function was expressed by a variational form. To obtain an explicit
  expression for the rate function, \cite{FGK} assumes a {\em local
  blow-up condition} which is not proven. On the other hand,
  \cite{Sch} obtains the explicit expression of the rate function
  when the term representing the movements of the particles also
  tends to zero. The local
  blow-up condition of \cite{FGK} was recently removed by Xiang for
  SBM with finite and infinite initial
  measure,\cite{Xia2},\cite{Xia} respectively,
   and the same explicit expression was established.
  Fleischmann and Xiong \cite{FX} proved an LDP for catalytic SBM with a single point catalyst.
  The successes of the LDP for SBM depend on the branching property of this process. This property
  implies the weak LDP directly, and hence the problem diminishes to showing the exponential tightness of SBM,
  which yields the LDP, and identifying its rate function.

    Since FVP does not possess the branching property, the derivation of LDP depends on new ideas. Dawson and
    Feng \cite{DF1}, \cite{DF2}, and Feng and Xiong \cite{FenX} considered the LDP for FVP
    when the mutation is neutral. In [4], LDP was shown to hold when the process remains in the interior of the simplex,
     and in [3] the authors proved that if the process starts from the interior, it will not reach the boundary.
     On the other hand, authors in \cite{FenX} focused on the singular case when the process starts from the boundary.
     For non-neutral case, Xiang and Zhang \cite{XZ}
     derived an LDP for FVP when the mutation operator also tends to zero by projecting to the finite dimensional case.

  The goal of this paper is to study LDP for MVP, with SBM and FVP as special cases. Comparing our LDP for
  SBM with that obtained in \cite{FGK}, \cite{Sch} and \cite{Xia}, the rate function has the same explicit
  representation, but the approach is different. Our LDP for FVP contributes to the literature, by not requiring the neutrality and vanishing of mutation.

  \section{Notations and Main Results}

 Let $\left(\Omega, \mathcal{F}, P, \mathcal{F}_{t} \right)$ be a stochastic basis satisfying the usual
 conditions of right continuity and completeness. Suppose $W$ is an $\mathcal{F}_{t}$-adapted space-time white noise random measure on
 $\mathbb{R}_{+} \times U$ with intensity measure $ds\lambda(da)$, where
  $\left(U, \mathcal{U}, \lambda\right)$ is a measure space.

   We consider the following stochastic partial differential equation (SPDE): for $\epsilon > 0$, $t\in [0,1]$, and $y\in \mathbb{R}$,
 \begin{equation}\label{2.1}
 u^{\epsilon}_{t}(y)= F(y) + \sqrt{\epsilon} \displaystyle
 \int_{0}^{t} \int_{U} G\left(a,y,u_{s}^{\epsilon}(y)\right) W(dsda) + \displaystyle \int_{0}^{t} \frac{1}{2}\Delta u_{s}^{\epsilon}(y)ds
 \end{equation}
 where $F$ is a function on $\mathbb{R}$ and $G: U\times \mathbb{R}^{2} \rightarrow \mathbb{R}$ satisfies the following conditions:
 there exists a constant $K>0$ such that for any $u_{1},u_{2}, u, y \in \mathbb{R}$,
 \begin{equation}\label{eq0502a}
 \displaystyle \int_{U} \left|G(a,y,u_{1})-G(a,y,u_{2})\right|^{2} \lambda(da) \leq K\left|u_{1}-u_{2}\right|
 \end{equation}
and \begin{equation}\label{eq0502b} \displaystyle \int_{U}
\left|G(a,y,u)\right|^{2} \lambda(da) \leq K\(1+|u|^2\).
\end{equation}
This SPDE was studied by Xiong \cite{Xio} in a Hilbert space denoted
by him as $\chi_{0}$. To study the LDP for the random field
$\left\{u_{t}^{\epsilon}(y)\right\}$, we need to consider the SPDE
in a certain H$\ddot{o}$lder continuous space; that is, we study the
regularity of the solution. For this purpose the spaces for the
solution are introduced. Let $\{\phi_{j}\}_{j\geq 1}$ be a complete
orthonormal system (CONS) of $L^{2}(U,\mathcal{U}, \lambda)$ and
define a system of stochastic processes as,
\begin{equation}\label{2.4}
B_{t}^{j} = \displaystyle \int_{0}^{t} \int_{U} \phi_{j}(a)
W(dsda),\qquad j=1,2,\cdots.
\end{equation}
which by L$\acute{e}$vy's characterization of Brownian motions, is a sequence of
independent Brownian motions. Denote the measurable space,
\begin{equation*}
(S,\mathbb{S}) := \left(\mathcal{C}([0,1]; \mathbb{R}^{\infty}), \mathbb{B}(\mathcal{C}([0,1];
\mathbb{R}^{\infty}))\right).
\end{equation*}
For any $\alpha \in (0,1)$ and $0<\beta\in\RR$, let the space,
$\mathbb{B}_{\alpha,\beta}$ be the collection of all functions $f:
\mathbb{R} \rightarrow \mathbb{R}$ such that for all $m\in \NN$,
\begin{equation}\label{eq3}
\left|f(y_{1})-f(y_{2})\right| \leq K e^{\beta m}
 \left|y_{1}-y_{2}\right|^{\alpha},
\hspace{.5cm} \forall |y_{1}|,\ |y_{2}|\le m,
 \end{equation} and \begin{equation}\label{eq0405a} \left|f(y)\right| \leq K
e^{\beta|y|}, \qquad \qquad \forall \ y\in \mathbb{R}.\end{equation}

We define the metric on $\BB_{\al,\be}$ as follows:
\[d_{\al,\be}(u,v)=\sum^\infty_{m=1}2^{-m}\(\|u-v\|_{m,\al,\be}\wedge
1\),\qquad u,\;v\in\BB_{\al,\be}\] where
\[\|u\|_{m,\al,\be}=\sup_{x\in\RR}e^{-\be|x|}|u(x)|+\sup_{y_{1}\neq y_{2}}\frac{|u(y_{1})-u(y_{2})|}{|y_{1}-y_{2}|^\al}e^{-\be m}.\]

Notice that SPDE (\ref{2.1}) can be rewritten as an SPDE driven by Brownian motions, $\{B_{t}^{j}\}$ as follows:
\begin{equation}\label{BSPDE}
u_{t}^{\epsilon}(y)= F(y) + \sqrt{\epsilon} \sum_{j} \int_{0}^{t}
G_j(y,u_{s}^{\epsilon}(y))dB_{s}^{j} +\displaystyle \int_{0}^{t}
\frac{1}{2}\Delta u_{s}^{\epsilon}(y)ds,
\end{equation}
where
\begin{equation}
G_j(y,u)=\int_U G(a,y,u)\phi_j(a)\la(da),\qquad j=1,2,\cdots.
\end{equation}

In this paper, we let $\beta_{0} \in (0,\beta)$.
\begin{theorem}\label{thm1}
For any $\alpha \in (0,\frac{1}{2})$, there exists a measurable map,
$g^{\epsilon}: \mathbb{B}_{\alpha,\beta_{0}} \times \mathbb{S}
\rightarrow \mathcal{C}([0,1]; \mathbb{B}_{\alpha,\beta})$ such that
for $F\in \mathbb{B}_{\alpha,\beta_{0}}$, $u^{\epsilon} =
g^{\epsilon}(F, \sqrt{\epsilon} B)$ is the unique mild solution of
(\ref{2.1}).
\end{theorem}

In order to study the LDP of the process $u_{t}^{\epsilon}$, one needs to consider the
controlled version of (\ref{2.1}) with the noise replaced by the control. For any $h \in L^{2}([0,1]\times
U, ds\lambda(da))$, this version has the following deterministic form,
\begin{equation}\label{2.6}
u_{t}(y) = F(y) + \displaystyle \int_{0}^{t} \int_{U}
G(a,y,u_{s}(y))h_{s}(a)\lambda(da)ds + \displaystyle \int_{0}^{t}
\frac{1}{2} \Delta u_{s}(y) ds
\end{equation}
Because of the non-Lipschitz continuity of the coefficients, the topology of the state space,
 $\mathcal{C}\left([0,1];\mathbb{B}_{\alpha,\beta}\right)$, needs to be modified.
\begin{definition}
We say that $u,v \in \mathcal{C}([0,1]; \mathbb{B}_{\alpha,\beta})$
are equivalent, denoted by $u\sim v$, if there exists an $h \in
L^{2}([0,1] \times U, ds \lambda(da))$ such that both $u,v$ are
solutions to (\ref{2.6}). If $u$ is not a solution to equation
(\ref{2.6}) for a suitable $h$, then $u$ belongs to the equivalent
class consisting of itself only.
\end{definition}
From this point on, we establish the LDP of $u^{\epsilon}$ in the quotient space of
 $\mathcal{C}([0,1]; \mathbb{B}_{\alpha,\beta})$ under the equivalence relation $\sim$ given above.
 We abuse the notation a bit by using the same notation for this quotient space. Note that when $h=0$, equation (\ref{2.6}) has a unique solution, $u_{t}^{0}(y)$. Therefore, this modification of topology does not affect the exponential rate of the form (\ref{1.1}) derived from the LDP at a neighborhood of $u^{0}_{.}$.

 Let $\gamma$ be a map from $\mathbb{B}_{\alpha,\beta_{0}}\times L^{2}([0,1] \times U, ds\lambda(da))$ to $\mathcal{C}([0,1];\mathbb{B}_{\alpha,\beta})$
 whose domain consists of $(F,h)$ such that (\ref{2.6}) has a solution, and denote the equivalence class of the solution as $u=\gamma(F,h)$.

 \begin{theorem}
 Suppose $F \in \mathbb{B}_{\alpha,\beta_{0}}$, then the family $\{u^{\epsilon}\}$ satisfies the LDP in $\mathcal{C}([0,1]; \mathbb{B}_{\alpha, \beta})$
 with rate function,
  \begin{equation}\label{2.7}
I(u) =\left\{\begin{array}{ll}  \frac{1}{2} \inf\left\{ \displaystyle
\int_{0}^{1} \int_{U} |h_{s}(a)|^{2} \lambda(da)ds : u = \gamma\left(F,h\right)\right\}
 &\exists h\; \mbox{s.t. } u=\gamma\left(F,h\right)\\
\infty &\mbox{otherwise.}
\end{array}\right.
\end{equation}
 \end{theorem}

We now apply Theorem 2 to SBM and FVP. Suppose $\{\mu^{\epsilon}\}$ is an SBM with branching rate $\epsilon$. As indicated by Xiong \cite{Xio}, for all $y\in \mathbb{R}$,
 \begin{equation}\label{SBME}
 u^{\epsilon}_{t}(y) = \int_{0}^{y}\mu_{t}^{\epsilon}(dx)
 \end{equation}
  is the unique solution to SPDE (\ref{2.1}) with
  \begin{equation}\label{2.8}
  F(y)=\int_{0}^{y}\mu_{0}(dx), \hspace{.2cm}U= \mathbb{R},\hspace{.2cm} \lambda(da)=da\hspace{.2cm}\mbox{ and } G(a,y,u)= 1_{a<u}.
  \end{equation}

Assume $\mathcal{D}$ is the Schwartz space of test functions with compact support in $\mathbb{R}$
and continuous derivatives of all orders. Denote the dual space of real distributions on
$\mathbb{R}$ by $\mathcal{D}^{*}$. Similar to \cite{FGK}, for a fixed $\nu \in \mathcal{M}_\be(\mathbb{R})$, let
the Cameron-Martin space, $H_{\nu}$, be the set of measures
 $\mu \in \mathcal{C}([0,1];\mathcal{M}_\be(\mathbb{R}))$ satisfying the conditions below.
\begin{enumerate}
\item  $\mu_{0}=\nu$,
\item the $\mathcal{D}^{*}$-valued map $t\mapsto \mu_{t}$ defined on [0,1] is absolutely continuous with
 respect to time. Let $\dot{\mu}$ and $\Delta^{*}\mu$ be its generalized derivative and Laplacian
  respectively,
\item for every $t\in [0,1]$, $\dot{\mu}_t - \frac{1}{2}\Delta^{*} \mu_{t} \in \mathcal{D}^{*}$ is
 absolutely continuous with respect to $\mu_{t}$
 with $\frac{d\left(\dot{\mu}_{t}-\frac{1}{2}\Delta^{*}\mu_{t}\right)}{d\mu_t}$
  being the (generalized) Radon Nikodym derivative,
\item $\frac{d(\dot{\mu}_{t} -\frac{1}{2}
\Delta^{*}\mu_{t})}{d\mu_{t}}$ is in $L^{2}([0,1] \times \mathbb{R},
ds\mu(dy))$.
\end{enumerate}

The topology of $\cM_\be(\RR)$ is defined by the following modified
weak convergence topology. We say that $\mu^n\to\mu$ in
$\cM_\be(\RR)$ if for any $f\in C_b(\RR)$,
\[\int_\RR f(x)e^{-\be|x|}\mu^n(dx)\to \int_\RR
f(x)e^{-\be|x|}\mu(dx).\]

  \begin{theorem}
If $\mu_{0} \in \mathcal{M}_\be(\mathbb{R})$ such that $F \in
\mathbb{B}_{\alpha,\beta_{0}}$, then $\{\mu^{\epsilon}\}$ satisfies
the LDP on $\mathcal{C}([0,1];\mathcal{M}_\be(\mathbb{R}))$ with
rate function,
 \begin{equation}\label{rate4sbm}
  I(\mu)=  \left\{\begin{array} {ll}  \frac{1}{2} \displaystyle \int_{0}^{1}
  \int_{\mathbb{R}}\left|\frac{\left(\dot{\mu}_{t} - \frac{1}{2} \Delta^* \mu_{t}\right)(dy)}{\mu_{t}(dy)}\right|^2 \mu_{t}(dy) dt
   & \mbox{\emph{if }} \mu \in H_{\mu_{0}} \\
   \infty
  & \mbox{\emph{otherwise}.}     \end{array}   \right.
 \end{equation}
\end{theorem}
 As for FVP, if $\left\{\mu^{\epsilon}\right\}$ is an FVP, then $u^\ep_t$ is defined as
\begin{equation*}
u^{\epsilon}_{t}(y) = \mu_{t}^{\epsilon} ((-\infty, y])
\end{equation*}
for all $y\in \mathbb{R}$, and by this definition, $u^{\epsilon}_{t}$ is the solution of SPDE (\ref{2.1}) with
\begin{equation}\label{2.9}
F(y)=\mu_{0}((-\infty,y]), \hspace{.2cm} U =[0,1], \hspace{.2cm}
\lambda(da)= da\hspace{.2cm}\mbox{ and } G(a,y,u)= 1_{a<u} -u.
\end{equation}
In this case, let $\tilde{H}_{\nu}$ be the space for which
conditions for $H_{\nu}$ hold with $\mathcal{M}_\be(\mathbb{R})$
replaced by $\mathcal{P}_{\beta}(\mathbb{R}):=
\mathcal{M}_{\beta}(\mathbb{R}) \cap \mathcal{P}(\mathbb{R})$, the
collection of Borel probability measures on $\mathbb{R}$, and with
the additional assumption,
\begin{equation*}
\left<\mu_{t}, \frac{\left(\dot{\mu}_{t}-\frac{1}{2} \Delta^*
\mu_{t}\right)(dy)}{\mu_{t}(dy)}\right> = 0.
\end{equation*}

\begin{theorem}
Suppose $\mu_{0} \in \mathcal{P}_{\beta}(\mathbb{R})$ such that $F
\in \mathbb{B}_{\alpha,\beta_{0}}$. Then, $\{\mu^{\epsilon}\}$
satisfies the LDP on $\mathcal{C}([0,1];
\mathcal{P}_{\beta}(\mathbb{R}))$ with rate function,
\begin{equation}\label{rate4fvp}
  I(\mu)=  \left\{\begin{array} {ll}  \frac{1}{2} \displaystyle \int_{0}^{1}
  \int_{\mathbb{R}}\left|\frac{\left(\dot{\mu}_{t} - \frac{1}{2} \Delta^{*} \mu_{t}\right)(dy)}{\mu_{t}(dy)}\right|^2 \mu_{t}(dy) dt
   & \mbox{\emph{if }} \mu \in \tilde{H}_{\mu_{0}} \\
   \infty
  & \mbox{\emph{otherwise.}}     \end{array}   \right.
 \end{equation}
 \end{theorem}

Proofs of Theorems 1-4 will be given in Sections 3-6. Throughout the
rest of this paper, $K$ will denote a constant whose value can be
changed from place to place.

\section{Regularity of SPDE}

This section is devoted to the proof of Theorem 1. For the simplicity of notation, we take $\epsilon = 1$
and denote $u_{t}^{\epsilon}(y)$ by $u_{t}(y)$. The solution to SPDE (\ref{2.1}) can then be written in the following mild form,
\begin{eqnarray}\label{3.1}
u_{t}(y) &=&  \int_{\mathbb{R}} p_{t}(y-x)F(x)dx \nonumber\\
&& + \displaystyle \int_{0}^{t} \int_{\mathbb{R}}\int_{U} p_{t-s}(y-x)G(a,x,u_{s}(x))W(dsda)dx
\end{eqnarray}
where $p_{t}(x) = \frac{1}{\sqrt{2\pi t}}
\exp\left(-\frac{x^2}{2t}\right)$ is the heat kernel. Here we refer
to the first term on the RHS of (\ref{3.1}) by $u^{0}_{t}(y)$, and
the second term by $v_{t}(y)$.

The following lemma offers an estimate that is the starting point
for other more refined estimates of the solution. The proof is
identical to that of Lemma 2.3 in \cite{Xio} so we omit it.

\begin{lemma}
For any $n\ge 2$ and $\be_1\in(\be_0,\be)$, we have
\begin{equation}\label{eq0502c}
M:=\sup_{0\le s\le 1}\EE\(\int_\RR
|u_s(x)|^2e^{-2\be_1|x|}dx\)^n<\infty.\end{equation}
\end{lemma}

Inspired by Shiga \cite{Shi}, to obtain the regularity of the  solution to SPDE (\ref{3.1})
and for its tightness to be used in a later section, the following
refined version of Kolmogorov's criterion is proved and applied.

\begin{lemma}\label{lem0405a}
Let $\{u^{\epsilon}_t(y)\}$ be a sequence of random fields and
suppose $\beta_1\in (\beta_{0},\beta)$. If there exist constants
$n,\;q,\;K>0$ such that
\begin{equation}\label{kolmogorov}
\mathbb{E}\left|u^{\epsilon}_{t_1}(y_1)-u^{\epsilon}_{t_2}(y_2)\right|^n
\le
Ke^{n\be_1(|y_1|\vee|y_2|)}\left(|y_1-y_2|+|t_1-t_2|\right)^{2+q},
\end{equation}
then,
\begin{equation}\label{Holder}
\sup_{\ep>0} \mathbb{E}\left|\sup_{m} \sup_{t_{i} \in [0,1],
|y_{i}|\leq m,i=1,2}
\frac{\left|u^{\epsilon}_{t_1}(y_1)-u^{\epsilon}_{t_2}(y_2)\right|}{\left(|y_1-y_2|+|t_1-t_2|\right)^{\alpha}}e^{-\beta
m}\right|^n < \infty.
\end{equation}
Furthermore, if
$\displaystyle\sup_{\ep>0}\mathbb{E}\left|u^{\epsilon}_{t_{0}}(y_{0})\right|^n
< \infty$ for some $(t_{0},y_{0}) \in [0,1]\times \mathbb{R}$, then
\begin{equation}\label{bdd}
\sup_{\ep>0}\EE\left|\displaystyle \sup_{(t,y)\in [0,1]\times
\mathbb{R}} e^{-\beta|y|} |u_{t}^{\epsilon}(y)|\right|^n<\infty.
\end{equation}
With the above additional assumption, the sequence
$\{u^{\epsilon}\}$ is tight in ${\cal
C}([0,1];\mathbb{B}_{\alpha,\beta})$.
\end{lemma}

\begin{proof}
For $i=1,2$, let $y'_{i}:=\frac1m y_i$ and
$\tilde{u}_{t}^{\epsilon}(y'_{i}):=u_{t}^{\epsilon}(y_{i})$. By the
hypothesis,
\begin{eqnarray}\label{K1}
\mathbb{E} \left|\tilde{u}_{t_{1}}^{\epsilon}(y'_{1})-\tilde{u}_{t_{2}}^{\epsilon}(y'_{2})\right|^{n} &=& \mathbb{E} \left|u_{t_{1}}^{\epsilon}(my'_{1})-u_{t_{2}}^{\epsilon}(my'_{2})\right|^{n}\nonumber\\
&\leq& Ke^{n\beta_{1}\left(|y_{1}|\vee |y_{2}|\right)} \left(m\left|y'_{1}-y'_{2}\right| + |t_{1}-t_{2}|\right)^{2+q}\nonumber\\
&\leq& Km^{2+q} e^{n\beta_{1}m} \left(\left|y'_{1}-y'_{2}\right| +
|t_{1}-t_{2}|\right)^{2+q}.
\end{eqnarray}
By Kolmogorov's criterion (cf. Corollary 1.2 in Walsh \cite{Wal}),
there exists a random variable $Y_{m}$ such that $
\mathbb{E}Y_{m}^{n} \leq K m^{2+q} e^{n\beta_{1}m}$ and
\begin{equation*}
\left|\tilde{u}_{t_{1}}^{\epsilon}(y'_{1})-\tilde{u}_{t_{2}}^{\epsilon}(y'_{2})\right|
\leq Y_{m}\left(|y'_{1}-y'_{2}| + |t_{1}-t_{2}|\right)^{q/n}
\end{equation*}
therefore,
\begin{equation}\label{eq0305a}
\left|u_{t_{1}}^{\epsilon}(y_{1})-u_{t_{2}}^{\epsilon}(y_{2})\right|
\leq Y_{m} \left(|y_{1}-y_{2}| + |t_{1}-t_{2}|\right)^{q/n}.
\end{equation}
Let $Y:= \displaystyle \sup_{m}\{ Y_{m}e^{-\beta m}\}$. Then,
\begin{eqnarray}\label{K2}
\mathbb{E}Y^{n} &\leq& \mathbb{E} \sum_{m} Y_{m}^{n} e^{-\beta mn}\\
&=& \displaystyle \sum_{m}\mathbb{E} Y_{m}^{n} e^{-\beta mn}\nonumber\\
&\leq& \sum_{m} K m^{2+q} e^{-\left(\beta-\beta_{1}\right)mn} <
\infty. \nonumber
\end{eqnarray}
Thus, $Y$ is a finite random variable, and (\ref{eq0305a}) implies
\begin{equation}\label{eq0305b}
\left|u_{t_{1}}^{\epsilon}(y_{1})-u_{t_{2}}^{\epsilon}(y_{2})\right|
\leq Y e^{\beta m}\left(|y_{1}-y_{2}| + |t_{1}-t_{2}|\right)^{q/n}.
\end{equation}

Now, we suppose there exists $(t_0,y_0) \in [0,1] \times \mathbb{R}$
such that
\[\sup_{\ep>0}\EE|u_{t_0}^{\epsilon}(y_0)|^n< \infty.\]
Note that (\ref{eq0305b}) remains true with $\be$ replaced by
$\be_2\in(\be_1,\be)$. For the simplicity of notation, we choose
$t_0=y_0=0$. Taking $t_1=t$, $y_1=y$ and $t_2=y_2=0$ in
(\ref{eq0305b}), gives
\[|u^\ep_t(y)|\le |u^\ep_0(0)|+Y e^{\beta_{2} m}\left(|y| +
|t|\right)^{q/n}.\] Suppose that $|y|\le m$. Then,
\begin{eqnarray*}
e^{-\be|y|}|u^\ep_t(y)|&\le& e^{-\beta|y|}|u^\ep_0(0)|+Y
e^{-(\beta-\be_2)
|y|}e^{\be_2}\left(|y| + |t|\right)^{q/n}\\
&\le&K\(e^{-\beta_0|y|}|u^\ep_0(0)|+Y\)
\end{eqnarray*}
for a suitable constant $K$ (independent of $m$). Inequality
(\ref{bdd}) then follows easily.

Uniform boundedness and equicontinuity are implied by (\ref{bdd})
and (\ref{Holder}), respectively. Therefore, tightness of the
sequence follows from Arzel\`a-Ascoli and Prohorov theorems.
 \end{proof}

The following lemmas illustrate $u_{t}^{0}$ and $v_{t}$ are included in $\mathbb{B}_{\alpha, \beta}$ space.
 These lemmas along with the result in \cite{Xio} on existence and uniqueness of a mild solution to SPDE (\ref{2.1}), prove Theorem 1.

\begin{lemma}\label{lem0420a}
 $u^{0}_{\cdot}$ is an element of $\mathcal{C}([0,1];\mathbb{B}_{\alpha, \beta})$.
\end{lemma}

\begin{proof}
Suppose $t\in [0,1]$ and $y_{1},y_{2}$ are any real numbers such
that $\left|y_{i}\right| \leq m$ for $i=1,2$. Let $B_{t}$ be a
Brownian motion. Then,
\begin{equation*}
u_{t}^{0}(y)= \mathbb{E} F(y-B_{t}).
\end{equation*}
Choosing $\gamma >0$ such that $(1+\gamma) \beta_{0} \leq \beta$ gives,
\begin{eqnarray}\label{est 1}
\left|u_{t}^{0}(y_{1}) - u_{t}^{0}(y_{2}) \right| &\leq& \mathbb{E}
\left|F(y_{1}-B_{t})- F(y_{2} - B_{t})\right| \nonumber\\
&=& \displaystyle \sum_{j=0}^{\infty} \mathbb{E}
\left|F(y_{1}-B_{t})-F(y_{2}-B_{t})\right| 1_{jm\gamma \leq
|B_{t}| \leq (j+1)\gamma {m}}\nonumber\\
&\leq& \displaystyle \sum_{j=0}^{\infty} K e^{\left((j+1)\gamma
+1\right)
\beta_{0}m} |y_{1}-y_{2}|^{\alpha} P(|B_{t}| \geq j\gamma m)\nonumber\\
&\leq& K \displaystyle \sum_{j=0} ^{\infty} e^{(j+1) \gamma
\beta_{0} m - \frac{1}{4} j^{2}m^{2}\gamma^{2} + \beta_{0}m}
|y_{1}-y_{2}|^{\alpha}\nonumber\\
&=& K e^{\beta_{0}m} \displaystyle \sum_{j=0}^{\infty} e^{\gamma
(j+1) \beta_{0} m - \frac{1}{4} j^{2} m^{2} \gamma^{2}}
|y_{1}-y_{2}|^{\alpha}\nonumber\\
&\leq& K e^{(1+\gamma)\beta_{0} m} m \gamma
|y_{1}-y_{2}|^{\alpha}\nonumber\\
&\leq& K e^{m\beta} |y_{1}-y_{2}|^{\alpha}
\end{eqnarray}
On the other hand, let $y$ in $\mathbb{R}$ be fixed such that $|y|\leq m$,
then for any $0<t_{1} \leq t_{2} <1$,
\begin{eqnarray}\label{est 2}
&& \left|u_{t_{1}}^{0}(y) - u_{t_{2}}^{0}(y)\right|\nonumber \\
&\leq& \mathbb{E} \left|F\left(y-B_{t_{1}}\right)-F\left(y-B_{t_{2}}\right)\right| \nonumber\\
&=& \displaystyle \sum_{j_{1},j_{2}} \mathbb{E} \left|F\left(y-B_{t_{1}}\right) - F\left(y-B_{t_{2}}\right)\right| 1_{j_{1}m\gamma \leq |B_{t_{1}}| \leq (j+1)m\gamma} 1_{j_{2}m\gamma \leq |B_{t_{2}}| \leq (j_{2}+1)m\gamma}\nonumber\\
&\leq& \displaystyle \sum_{j_{1},j_{2}} K e^{(m+(j_{1}\vee j_{2} +1)m\gamma)\beta_{0}} \mathbb{E}\left(\left|B_{t_{1}}-B_{t_{2}}\right|^{\alpha} 1_{|B_{t_{1}}| \geq j_{1}m\gamma} 1_{|B_{t_{2}}| \geq j_{2}m\gamma}\right)\nonumber\\
&\leq&  \displaystyle \sum_{j_{1},j_{2}} K e^{(j_{1}\vee j_{2} +1)m\gamma \beta_{0}+m\beta} (\mathbb{E}|B_{t_{1}}-B_{t_{2}}|^{2\alpha})^{\frac{1}{2}} P\left(|B_{t_{1}} \geq j_{1}m\gamma, |B_{t_{2}}| \geq j_{2}m\gamma\right)^{\frac{1}{2}}\nonumber\\
&\leq& \displaystyle \sum_{j_1,j_2}  K e^{((j_1\vee j_2+1)\gamma
+1)\beta_0m}
|t_1-t_2|^{\alpha/2} e^{-\frac{1}{4}m^2\gamma^2(j_1^2 + j_2^2)}\nonumber\\
&\leq& K(m\gamma)^2 e^{\beta_{0}(1+\gamma)m} |t_{1}-t_{2}|^{\alpha/2}\nonumber\\
&\leq& K e^{m\beta} |t_{1}-t_{2}|^{\alpha/2}.
\end{eqnarray}
Estimates (\ref{est 1}) and (\ref{est 2}) imply that
$u^{0}_{\cdot}\in\mathcal{C}([0,1];\mathbb{B}_{\alpha, \beta})$.
\end{proof}

\begin{lemma}\label{lemv}
$v_{\cdot}$ takes values in
$\mathcal{C}\left([0,1];\mathbb{B}_{\alpha,\beta}\right)$, a.s.
\end{lemma}

\begin{proof}
Similar to the proof of Lemma \ref{lem0420a}, two cases are
demonstrated for this lemma. Considering the first case, denote
\begin{equation*}
G:= G(a,x,u_{s}(x)) \mbox{ and } P_{1}:= p_{t-s}(y_{1}-x)-p_{t-s}(y_{2}-x)
\end{equation*}
and let $t\in [0,1]$ be fixed, while $y_{1},y_{2} \in \mathbb{R}$ are arbitrary numbers such that $|y_{i}|\leq m$ for $i=1,2$. Using Burkholder-Davis-Gundy and H$\ddot{o}$lder's inequalities, we obtain,
\begin{eqnarray}\label{3.2}
&&\mathbb{E}\left|v_{t}(y_{1})-v_{t}(y_{2})\right|^{n} \nonumber \\
&=& \mathbb{E}\left|\displaystyle \int_{0}^{t}
\int_{\mathbb{R}}\int_{U} P_{1} G dx W(dsda)\right|^n \nonumber\\
 &\leq& K\mathbb{E}\left|\int_{0}^{t}\int_{U}\left|\int_{\mathbb{R}}
P_{1} G dx\right|^2\lambda(da)ds\right|^{n/2}\\
\nonumber &\leq& K \mathbb{E}\left|\displaystyle \int_{0}^{t}
\int_{U} \int_{\mathbb{R}} \left|P_{1}\right|^{2}
e^{2\beta_{1}|x|}dx \int_{\mathbb{R}}G^{2} e^{-2\beta_{1}|x|} dx
\lambda(da)ds \right|^{n/2}\\ \nonumber &\leq& K \mathbb{E}
\left|\displaystyle \int_{0}^{t} J_{t-s}(y_{1},y_{2}) \displaystyle
\int_{\mathbb{R}}\(1+|u_{s}(x)|^2\) e^{-2\beta_{1}|x|}
dxds\right|^{n/2}
\end{eqnarray}
 where
\begin{equation*}
J_{s}(y_{1},y_{2}) = \displaystyle \int_{\mathbb{R}}
\left|p_{s}(y_{1}-x) - p_{s}(y_{2}-x)\right|^{2} e^{2\beta_{1}
|x|}dx
\end{equation*}
is estimated below using the simplified notation,  \[P_{2} := p_{s}(y_{1}-x)-p_{s}(y_{2}-x).\]
 \begin{eqnarray}\label{eq0502d}
 J_{s}(y_{1},y_{2}) &=& \int_{\mathbb{R}} \left|P_{2}\right|^{\alpha} \left|P_{2}\right|^{2-\alpha} e^{2\beta_{1}|x|}dx \\
&\leq& \int_{\mathbb{R}} \left|\frac{1}{\sqrt{2\pi s}}\right|^{\alpha} \left|\frac{(y_{1}-x)^{2} - (y_{2}-x)^2}{2s}\right|^{\alpha}
 \left|P_{2}\right|^{2-\alpha}e^{2\beta_{1}|x|}dx\nonumber\\
&\leq& K\int_{\mathbb{R}} \left|\frac{1}{\sqrt{2\pi s}}\right|^{\alpha} \frac{|y_{1}-y_{2}|^{\alpha}
 |y_{1}+y_{2}-2x|^{\alpha}}{(2s)^{\alpha}(2\pi s)^{(2-\alpha)/2}} e^{-\frac{(2-\alpha)(y_{1}-x)^2}{2s}} e^{2\beta_{1}|x|} dx \nonumber\\
 && +K\int_{\mathbb{R}} \left|\frac{1}{\sqrt{2\pi s}}\right|^{\alpha} \frac{|y_{1}-y_{2}|^{\alpha}
  |y_{1}+y_{2}-2x|^{\alpha}}{(2s)^{\alpha}(2\pi s)^{(2-\alpha)/2}}  e^{-\frac{(2-\alpha)(y_{2}-x)^{2}}{2s}} e^{2\beta_{1}|x|} dx\nonumber\\
&\leq& K|y_{1}-y_{2}|^{\alpha} s^{-(1+\alpha)}\displaystyle
\int_{\mathbb{R}}  |y_{1}+y_{2}-2x|^{\alpha}
 e^{-\frac{(2-\alpha)(y_{1}-x)^{2}}{2s}}e^{2\beta_{1}|x|}dx\nonumber\\
&& +K|y_{1}-y_{2}|^{\alpha} s^{-(1+\alpha)}\displaystyle
\int_{\mathbb{R}}
 |y_{1}+y_{2}-2x|^{\alpha}
 e^{-\frac{(2-\alpha)(y_{2}-x)^{2}}{2s}}e^{2\beta_{1}|x|}dx\nonumber \\
&\leq& K e^{2\beta_{1}(|y_{1}|\vee |y_{2}|)} s^{-(\frac{1}{2} +
\alpha)} |y_{1}-y_{2}|^{\alpha}.\nonumber
\end{eqnarray}
Note that, we may choose $p>1$ such that $\(\frac12+\al\)p<1$ and
let $q$ be the conjugate index. Then,
\begin{eqnarray}\label{eq0502e}
&&\EE\(\int^t_0(t-s)^{-\(\frac12+\al\)}\int_\RR\(1+|u_s(x)|^2\)e^{-2\be_1|x|}dx
ds\)^{n/2}\nonumber\\
&\le&\(\int^t_0(t-s)^{-\(\frac12+\al\)p}ds\)^{n/(2p)}\EE\(\int^t_0\(\int_\RR\(1+|u_s(x)|^2\)e^{-2\be_1|x|}dx\)^q
ds\)^{n/(2q)}\nonumber\\
&\le&K\EE\int^t_0\(\int_\RR\(1+|u_s(x)|^2\)e^{-2\be_1|x|}dx\)^{n/2}ds\nonumber\\
&\le&K.
\end{eqnarray}
Plugging (\ref{eq0502d}) back into (\ref{3.2}) and noting
(\ref{eq0502e}) to obtain,
\[
\mathbb{E}|v_t(y_1)-v_t(y_2)|^n \le K e^{n\beta_{1}(|y_{1}|\vee
|y_{2}|)} |y_{1}-y_{2}|^{\frac{\alpha n}{2}}.\]

Next to prove case two, let $y \in \mathbb{R}$ and choose any $0\leq t_{1} < t_{2} \leq 1$. Note that
\begin{eqnarray}\label{Is}
&&\mathbb{E}\left|v_{t_{1}}(y)-v_{t_{2}}(y)\right|^n\\
&\leq& K \mathbb{E} \left|\displaystyle \int_{0}^{t_{1}} I_{s}(t_{1},t_{2}) \int_{\mathbb{R}} \(1+|u_{s}(y)|^2\) e^{-2\beta_{1}|x|} dxds\right|^{n/2}\nonumber\\
&& + K \mathbb{E} \left|\displaystyle \int_{t_{1}}^{t_{2}}
\int_{\mathbb{R}} p_{t_{2}-s}^{2} (y-x) e^{2\beta_{1}|x|} dx
\int_{\mathbb{R}} \(1+|u_{s}(x)|^2\)e^{-2\beta_{1}|x|}
dxds\right|^{n/2}\nonumber
\end{eqnarray}
where
\begin{eqnarray*}
I_{s}(t_{1},t_{2})&:=& I^{1}_{s}(t_{1},t_{2}) + I^{2}_{s}(t_{1},t_{2}),\\
I^{i}_{s}(t_{1},t_{2})&:=& \int_{\mathbb{R}}
\left|p_{t_{1}-s}(y-x)-p_{t_{2}-s}(y-x)\right|^{\alpha}p_{t_{i}-s}(y-x)^{2-\alpha}
e^{2\beta_{1}|x|}dx
\end{eqnarray*}
for $i=1,2$. We estimate $I^{1}_{s}(t_{1},t_{2})$ by $K\left(I^{11}_{s}(t_{1},t_{2}) + I^{12}_{s}(t_{1},t_{2})\right)$ where
\begin{equation*}
I^{11}_{s}(t_{1},t_{2}):= \displaystyle \int_{\mathbb{R}}
\left|\frac{1}{\sqrt{t_{1}-s}} -
\frac{1}{\sqrt{t_{2}-s}}\right|^{\alpha}p_{t_{1}-s}(y-x)^{2-\alpha}
e^{2\beta_{1}|x|} dx
\end{equation*}
and
\begin{equation*}
I^{12}_{s}(t_{1},t_{2}):= \displaystyle \int_{\mathbb{R}}
\left|\frac{1}{\sqrt{t_{2}-s}} \left|\frac{1}{t_{1}-s} -
\frac{1}{t_{2}-s}\right|(y-x)^{2}\right|^{\alpha}
p_{t_{1}-s}(y-x)^{2-\alpha}e^{2\beta_{1}|x|}dx
\end{equation*}
Now we continue with
\begin{eqnarray*}
I^{11}_{s}(t_{1},t_{2}) &\leq& K \displaystyle \int_{\mathbb{R}} \left|\frac{t_{2}-t_{1}}{\sqrt{t_{1}-s}(t_{2}-s)} \right|^{\alpha} \frac{e^{2\beta_{1}|x|}}{\sqrt{t_{1}-s}^{1-\alpha}} p_{(t_{1}-s)/(2-\alpha)}(y-x) dx\\
&\leq& K
\frac{|t_{1}-t_{2}|^{\alpha}}{\sqrt{t_{1}-s}(t_{2}-s)^{\alpha}}
e^{2\beta_{1}|y|}
\end{eqnarray*}
and
\begin{eqnarray*}
I^{12}_{s}(t_{1},t_{2})&\leq& K \displaystyle \int_{\mathbb{R}} \frac{|t_{2}-t_{1}|^{\alpha}(y-x)^{2\alpha}}{(t_{2}-s)^{\frac{3\alpha}{2}}(t_{1}-s)^{\frac{1-\alpha}{2}}} p_{(t_{1}-s)/(2-\alpha)}(y-x) e^{2\beta_{1}|x|}dx \\
&\leq& K
\frac{(t_{2}-t_{1})^{\alpha}}{(t_{2}-s)^{\frac{3\alpha}{2}}(t_{1}-s)^{\frac{1-\alpha}{2}}}
e^{2\beta_{1}|y|}
\end{eqnarray*}
Recall $0\leq t_{1}<t_{2}\leq 1$ so for $\alpha \in (0,\frac{1}{2})$,
\begin{eqnarray*}
\displaystyle \int_{0}^{t_{1}} I^{11}_{s}(t_{1},t_{2})ds &\leq& K e^{2\beta_{1}|y|}|t_{1}-t_{2}|^{\alpha}\displaystyle \int_{0}^{t_{1}} (t_{1}-s)^{-(\frac{1}{2}+\alpha)} ds \\
&\leq& K e^{2\beta_{1}|y|} |t_{1}-t_{2}|^{\alpha}
\end{eqnarray*}
and
\begin{eqnarray*}
\displaystyle \int_0^{t_1} I_s^{12}(t_1,t_2)ds &\leq& K
e^{2\beta_1|y|}|t_{1}-t_{2}|^{\alpha}\displaystyle \int_{0}^{t_{1}}
(t_{1}-s)^{-(\frac{1}{2}+\alpha)} ds
\\
&\leq& K e^{2\beta_{1}|y|} |t_{1}-t_{2}|^{\alpha},
\end{eqnarray*}
where we used the fact that $t_{2}-s > t_{1}-s$. Making use of
(\ref{eq0502c}), we see that the first term of (\ref{Is}) can be
estimated above by
\begin{eqnarray*}
K\left(\displaystyle \int_{0}^{t_{1}} \left(I^{11}_{s}(t_{1},t_{2})  +  I^{12}_{s}(t_{1},t_{2}) + I^{21}_{s}(t_{1},t_{2}) + I^{22}_{s}(t_{1},t_{2})\right)ds\right)^{n/2}\\
\leq K e^{n\beta_{1}|y|} |t_{1}-t_{2}|^{\frac{\alpha n}{2}}
\end{eqnarray*}
where $I^{21}$ and $I^{22}$ are defined and estimated similarly as those for $I^{11}$ and $I^{12}$. \\
Finally, we consider the second term of (\ref{Is}). Notice that,
\begin{eqnarray*}
\displaystyle \int_{t_{1}}^{t_{2}} \int_{\mathbb{R}} p_{t_{2}-s}^{2}(y-x) e^{2\beta_{1}|x|} dxds &\leq& K \displaystyle \int_{t_{1}}^{t_{2}} \int_{\mathbb{R}} \frac{e^{2\beta_{1}|x|}}{\sqrt{t_{2}-s}} p_{\frac{1}{2}(t_{2}-s)}(y-x) dxds\\
&\leq& K \displaystyle e^{2\beta_{1}|y|}\int_{t_{1}}^{t_{2}} \frac{ds}{\sqrt{t_{2}-s}} \\
&\leq& K|t_{1}-t_{2}|^{\alpha/2}e^{2\beta_{1}|y|}
\end{eqnarray*}
Thus, we see that the second term of (\ref{Is}) is bounded by
\begin{equation*}
K e^{n\beta_{1}|y|}|t_{1}-t_{2}|^{\frac{n\alpha}{4}}
\end{equation*}
\end{proof}

\section{LDP for SPDE}
LDP describes the asymptotic behavior of the sequence $\{u^{\epsilon}\}$ of the above SPDE as
 $\epsilon \rightarrow 0$. This principle gives the following two bounds.

 LDP Lower bound: For all open sets, $\mathcal{U} \subset\mathcal{C}([0,1]; \mathbb{B}_{\alpha,\beta})$,
 \begin{equation*}
 \displaystyle \liminf_{\epsilon \rightarrow 0} \epsilon \log P(u^{\epsilon} \in \mathcal{U}) \geq -\displaystyle \inf_{x\in \mathcal{U}} I(x)
 \end{equation*}

 LDP Upper bound: For every closed set $C\subset \mathcal{C}([0,1]; \mathbb{B}_{\alpha,\beta})$,
 \begin{equation*}
 \displaystyle \limsup_{\epsilon \rightarrow 0} \epsilon \log P(u^{\epsilon} \in C) \leq -\inf_{x\in C} I(x)
 \end{equation*}
 where $I:\ \mathcal{C}([0,1]; \mathbb{B}_{\alpha,\beta}) \rightarrow [0,\infty]$ is a lower semicontinuous map called a rate function. For more introduction to the theory of large deviations, we refer the reader to \cite{Dem},\cite{Deu} and \cite{Dup}.
In this section, we derive the LDP for SPDE (\ref{2.1}) by using the
powerful technique developed
 by Budhiraja $et\; al$ \cite{BDM}. More specifically, we apply Theorem 6 of that paper
 with $\mathcal{E}_{0} := \mathbb{B}_{\alpha, \beta_{0}}$ and
 $\mathcal{E}:= \mathcal{C}([0,1]; \mathbb{B}_{\alpha,\beta})$.

 Recall from Section 2, the definition of the map $\gamma$ and let $g^{\epsilon}$ be the map
 given in Theorem 1. Denote
 \begin{equation*}
 \mathcal{S}^{N}(\ell_{2}):= \left\{k\in L^{2}([0,1]:\ell_{2}):
 \displaystyle \int_{0}^{1} \|k_s\|^{2}_{\ell_{2}}ds\leq N\right\}
 \end{equation*}
 and define a map $\zeta$ from $k\in\mathcal{S}^{N}(\ell_{2})$ to $h=\ze(k)\in L^{2}([0,1] \times U)$ as follows:
\begin{equation*}
h_{s}(a) =  \sum_{j}  k_{s}^{j} \phi_{j}(a).
\end{equation*}
Let $g^{0}:\mathbb{B}_{\alpha,\beta_{0}} \times S \rightarrow
\mathcal{C}([0,1]:\mathbb{B}_{\alpha,\beta})$ given by,
\begin{equation}
g^{0}\left(F,\displaystyle \int_{0}^{.}k_{s}ds\right) =
\gamma\left(F, \zeta \left(\displaystyle k\right)\right).
\end{equation}
To obtain the LDP, it is sufficient to verify Assumption 2 imposed by
\cite{BDM}. Suppose $\{k^\ep\}$ is a family of random variables taking
values in $\cS^N(\ell_2)$ such that $k^{\epsilon} \rightarrow k$ in
distribution and $F^{\epsilon} \rightarrow F$ as $\ep\to 0$. Denote
the solution to
\begin{eqnarray}\label{me}
u_{t}(y) &=& \displaystyle \int_{\mathbb{R}} p_{t}(y-x)F^{\epsilon}(x)dx +
\theta \displaystyle \sum_{j} \displaystyle \int_{0}^{t} \int_{\mathbb{R}} p_{t-s}(y-x)G_{j}(x,u_s(x)) dB^{j}_{s}dx \nonumber \\
&& + \displaystyle \sum_j\int_{0}^{t}\int_{\mathbb{R}}
p_{t-s}(y-x)G_{j}(x,u_s(x))k_{s}^{\epsilon,j}dx ds
\end{eqnarray}
 as $u_{t}^{\theta,\epsilon}(y)$, where $G_{j}(y,u)$ is defined in Section 2.

\begin{lemma}
$\{u^{\theta,\epsilon}\}$ is tight in $\mathcal{C}\left([0,1];\mathbb{B}_{\alpha,\beta}\right)$. In particular, Assumption 2 of \cite{BDM} holds under the current setup.
\end{lemma}

\begin{proof}
 To prove the tightness of $\{u^{\theta,\epsilon}\}$, we need to determine estimates for $u^{\theta,\epsilon}$ similar to those obtained in Section 3. Since the main difference is in the last term, we restrict our attention to
\begin{equation*}
w_{t}(y):= \sum_j\int_{0}^{t}\int_{\mathbb{R}}
p_{t-s}(y-x)G_{j}(x,u_s(x))k_{s}^{\epsilon,j}dx ds.
\end{equation*}
Using $P_{1}:= p_{t-s}(y_{1}-x)-p_{t-s}(y_{2}-x)$,
\begin{eqnarray*}
\mathbb{E}\left|w_{t}(y_{1})-w_{t}(y_{2})\right|^{n} &=&
\mathbb{E}\left|\int_{0}^{t}\int_{\mathbb{R}}
P_{1}\sum_{j}G_{j}(x,u_{s}(x))k_{s}^{\epsilon,j}dxds\right|^{n}\\
&\leq& \mathbb{E}\left|\int_{0}^{t}\int_{\mathbb{R}}|P_1|
\left(\sum_{j}G_{j}(x,u_{s}(x))^{2}\right)^{1/2}\|k_{s}^{\epsilon}\|_{\ell_{2}}dxds\right|^{n}\\
&\leq& \mathbb{E} \left|\int_{0}^{t}\left(\int_{\mathbb{R}}
|P_{1}|\sqrt{K\(1+|u_{s}(x)|^2\)} dx\right)^{2}ds\right|^{n/2}N^{n/2} \\
&\leq& MN^{n/2}\( \int_{0}^{t}\int_{\mathbb{R}}|P_{1}|^2e^{2\beta_{1}|x|}dxds\)^{n/2}\\
&\leq& Ke^{n\beta_{1}(|y_{1}|\vee |y_{2}|)}|y_{1}-y_{2}|^{\alpha}
\end{eqnarray*}
where the last step follows from an analogous argument as in the
proof of Lemma \ref{lemv} and $M$ is given by (\ref{eq0502c}). The
estimate for fixed $y$ and $t_1,\; t_2$ arbitrary can be derived
similarly. Now the first condition in Assumption 2 follows from the
above argument by taking $\theta = 0$, while the second condition is
verified by taking $\theta = \sqrt{\epsilon}$.
\end{proof}

 Suppose $u^{0,0}$ is a limit point of $\{u^{\theta,\epsilon}\}$ as $\theta,\epsilon \rightarrow 0$. By taking limits on both sides of (\ref{me}),  $u^{0,0}$ becomes a solution to the following equation,
\begin{eqnarray*}
u_{t}^{0,0}(y) &=&  \int_{\mathbb{R}} p_{t}(y-x)F(x)dx\\
&& +\sum_j \displaystyle \int_{0}^{t}\int_{\mathbb{R}}
p_{t-s}(y-x)G_{j}(x,u^{0,0}_s(x)) k^j_{s}dxds\nonumber\\
&=&\int_{\mathbb{R}} p_{t}(y-x)F(x)dx \\
&&+ \displaystyle \int_{0}^{t}\int_{\mathbb{R}} \int_{U}
p_{t-s}(y-x)G(a,x,u^{0,0}_s(x)) h_{s}(a)\lambda(da)dx ds
\end{eqnarray*}
which is the mild form of (\ref{2.6}), where $h=\zeta(k)$. The
definition of $\gamma$ implies $u^{0,0}=\gamma(F,h)$. Thus, using
the above lemma to apply Theorem 6 in [1], the rate function for
SPDE (\ref{2.1}) is given as,
\begin{equation*}
\tilde{I}(u) =\left\{\begin{array}{ll}  \frac{1}{2} \inf\left\{ \displaystyle
\int_{0}^{1}  \|k_{s}\|_{\ell_{2}}^{2} ds : u = \gamma\left(F,\zeta(k)\right)\right\}
 &\exists k\; \mbox{s.t.} u=\gamma\left(F,\zeta(k)\right)\\
\infty &\mbox{otherwise.}
\end{array}\right.
\end{equation*}

By the relationship between $k$ and $h$, it is easy to see that
$\tilde{I}$ coincides with rate function $I$ defined by
(\ref{2.7}). This concludes the proof of Theorem 2.

Function $G(a,x,u)$ for SBM and FVP satisfies conditions (\ref{eq0502a}) and
(\ref{eq0502b}); hence, by the results attained in this section, to establish
the LDP for SBM and FVP, one needs only to determine the corresponding
rate functions. The identification of these rate functions is given
in Sections 5 and 6 respectively.

\section{LDP for super-Brownian Motion}

SBM is one of the main models in studying the
evolution of populations. It assumes that each individual moves
randomly according to a Brownian motion and she leaves a random
number of offsprings upon her death. Therefore, SBM is a measure-valued process with an associated branching
rate, $\epsilon$. Formally speaking, this measure-valued process,
also referred to as a superprocess, is defined as the unique
solution, $\mu_{t}^{\epsilon}$, to the following martingale problem:
for all $f\in \mathcal{C}_{b}^{2}(\mathbb{R})$,
\begin{equation*}
M_{t}(f):= \left<\mu_{t}^{\epsilon},f\right> -
\left<\mu_{0},f\right> - \displaystyle \int_{0}^{t}
\left<\mu_{s}^{\epsilon},\frac{1}{2} \Delta f\right> ds,
\end{equation*}
is a square-integrable martingale with quadratic variation,
\begin{equation*}
\left<M_{t}(f)\right> = {\epsilon}\displaystyle \int_{0}^{t}
\left<\mu_{s}^{\epsilon},f^2\right> ds.
\end{equation*}
For more information on this superprocess see \cite{Eth}
and \cite{Li}.
 Our aim in this section is to prove the LDP for
SBM as the branching rate $\epsilon$ is set to
converge to zero.
We define
\begin{equation}\label{J}
J_{\beta}(x)=\int_{\mathbb{R}}e^{-\beta|y|}\rho(x-y)dy
\end{equation}
where $\rho$ is the mollifier given by
\begin{equation*}
\rho(x)=K \exp\left(\frac{-1}{1-x^2}\right)1_{|x|<1}
\end{equation*}
and $K$ is a constant such that $\int_{\mathbb{R}} \rho(x)dx=1$.
Then for all $m\in \mathbb{Z}_+$, there are constants $c_{m},C_{m}$
such that
\begin{equation*}
c_{m}e^{-\beta|x|} \leq J_{\beta}^{(m)}(x) \leq C_{m}e^{-\beta |x|}\hspace{.5cm} \forall x\in \mathbb{R}
\end{equation*}
(cf. Mitoma \cite{Mit}, (2.1)). Therefore, we may and will replace $e^{-\beta|x|}$ by $J_{\beta}(x)$ in the definition of $\mathcal{M}_{\beta}(\mathbb{R})$ given by (\ref{e}).

\begin{lemma}\label{lem0420b}
 Let $\mathcal{A}$ be the set of all nondecreasing functions, then the map $\xi: \mathbb{B}_{\alpha,\beta} \cap \mathcal{A} \rightarrow \mathcal{M}_{\beta}(\mathbb{R})$ defined as $\xi(u)(A)= \int 1_A (y) du(y)$ for all $A\in \mathcal{B}(\mathbb{R})$, is continuous.
 \end{lemma}

 \begin{proof}
 Suppose $u_{n} \rightarrow u$ in the space $\mathbb{B}_{\alpha,\beta} \cap \mathcal{A}$. Then for every $f\in \mathcal{C}_{b}(\mathbb{R})$,
 \begin{eqnarray*}
 \int f(x)J_{\beta}(x)\xi(u_n)dx &=& \int f(x)J_{\beta}(x)du_{n}(x)\\
 &=& -\int (fJ_{\beta})'(x)u_{n}(x)dx\\
 &\rightarrow& -\int (fJ_{\beta})'(x)u(x)dx\\
 &=& \int fJ_{\beta}(x)\xi(u)(dx)
 \end{eqnarray*}
  verifying the continuity of $\xi$ map.
 \end{proof}
 $\emph{Proof of Theorem 3}$ Recall the definition of $u^{\epsilon}$ given by (\ref{SBME}).
 By Theorem 2, $u^{\epsilon}$ satisfies the LDP on $\mathcal{C}([0,1]; \mathbb{B}_{\alpha, \beta})$
  and because $u_{t}^{\epsilon} \in \mathcal{A}$ a.s. for all $t$, we see that $u^{\epsilon}$ obeys LDP on
  $\mathcal{C}([0,1]; \mathbb{B}_{\alpha, \beta}\cap \mathcal{A})$, as well.
  Since for SBM, $\mu_{t}^{\epsilon} = \xi(u_{t}^{\epsilon})$, then
  by Lemma \ref{lem0420b} and the contraction principle, LDP holds for $\mu^{\epsilon}$ on
  $\mathcal{C}([0,1];\mathcal{M}_{\beta}(\mathbb{R}))$ with the rate function determined below.

If $I(\mu)<\infty$, then there exists
$h\in L^2([0,1]\times \mathbb{R}_{+}, ds da)$ such that (\ref{2.6})
holds. Let $\mathcal{C}_{c}(\mathbb{R})$ be the collection of functions with compact support on $\mathbb{R}$, then for $f\in \mathcal{C}_{c}^{1}(\mathbb{R})$,
\begin{equation*}
\<\mu_{t},f\> = - \< u_{t}, f'\>_{L^{2}(\mathbb{R})}.
\end{equation*}
Using the controlled version SPDE (\ref{2.6}), for every $f\in \mathcal{C}_{c}^{3}(\mathbb{R})$,
\begin{eqnarray*}
\<\mu_t,f\>&=&-\<F,f'\>-\int^t_0\<u'_s,\frac{1}{2}
\Delta f\>ds\\
&& \hspace{.3cm} -\int^t_0\int_{\mathbb{R}}\int_{-\infty}^{u_{s}(y)}
h_{s}(a)da f'(y)dy ds\\
&=&\<\mu_0,f\>+\int^t_0\<\mu_s,\frac12\De
f\>ds+\int^t_0\int_{\mathbb{R}}
f(u^{-1}_s(a))h_s(a) da  ds\\
&=&\<\mu_0,f\>+\int^t_0\<\frac12\De^*
\mu_s,f\>ds+\int^t_0\int_{\mathbb{R}}
f(y)h_s(u_s(y))du_s(y) ds\\
&=&\<\mu_0,f\>+\int^t_0\<\frac12\De^{*}\mu_s,
f\>ds+\int^t_0\<\mu_s,fh_s(u_s)\>ds
\end{eqnarray*}
which implies $\mu\in H_{\mu_0}$ and
\begin{equation*}
\frac{\left(\dot{\mu_{t}}- \frac{1}{2} \Delta^{*}
\mu_{t}\right)(dy)}{\mu_{t}(dy)} = h_{t}(u_{t}(y)).
\end{equation*}
Moreover,
\begin{equation*}
\int_{\mathbb{R}} \left|h_{t}(u_{t}(y))\right|^{2} \mu_{t}(dy) =
\int_{\mathbb{R}} \left|h_{t}(u_{t}(y))\right|^{2} du_{t}(y) =
\int_{\mathbb{R}} \left|h_{t}(a)\right|^{2} da.
\end{equation*}
Denote the right hand side of (\ref{rate4sbm}) by $I_0(\mu)$ and observe that in this case, $I_0(\mu)=I(\mu)$.

If $I_0(\mu)<\infty$, we may reverse the above
calculation to obtain the finiteness of $I(\mu)$. This completes the
proof of Theorem 3.

\section{LDP for Fleming-Viot}

Besides SBM, FVP is another important model used in population
evolution. In this model population size stays fixed throughout time
and the gene mutation and selection of individuals are observed. A
rigorous definition for Fleming-Viot process is a probability
measure-valued process $\mu_{t}^{\epsilon}$ solving the following
martingale problem: for all $f\in \mathcal{C}_{c}^{2}(\mathbb{R})$,
\begin{equation*}
N_{t}(f):=\left<\mu_{t}^{\epsilon},f\right> - \left<\mu_{0},f\right>
- \displaystyle \int_{0}^{t} \left<\mu_{s}^{\epsilon},\frac{1}{2}
\Delta f\right> ds
\end{equation*}
is a continuous square-integrable martingale with quadratic
variation,
\begin{equation*}
\left<N_{t}(f)\right> = {\epsilon}\displaystyle \int_{0}^{t}
\left(\left<\mu_{s}^{\epsilon}, f^{2}\right> -
\left<\mu_{s}^{\epsilon}, f\right>^{2} \right)ds.
\end{equation*}
More detailed material on Fleming-Viot process can be found in
\cite{Daw} and \cite{Eth}.
This section derives the LDP for
Fleming-Viot process as its mutation rate, $\epsilon$ is set to
converge to zero.

$\emph{Proof of Theorem 4}$ Using the same argument as in Section 5,
 Lemma \ref{lem0420b} can be proven for FVP by defining a map $\psi: \mathbb{B}_{\alpha,\beta} \cap \mathcal{A} \rightarrow \mathcal{P}_{\beta}(\mathbb{R})$ defined as
  $\psi(u)((-\infty,y])= \int_{-\infty}^{y}du(y)$. The continuity of this map can be easily verified following the same steps as in
  Lemma \ref{lem0420b}; therefore, we proceed by identifying the rate function.

If $I(\mu)<\infty$, there exists
$h\in L^2([0,1]\times\mathbb{R}_{+}, ds da)$ such that (\ref{2.6})
holds. For $f\in \mathcal{C}_{c}^{3}(\mathbb{R})$,
\begin{eqnarray*}
\<\mu_t,f\>&=&-\<F,f'\>-\int^t_0\<u'_s,\frac12
\Delta f\>ds\\
&& \hspace{.3cm} - \int^t_0\int_{\mathbb{R}}\int_{0}^{u_{s}(y)}
h_{s}(a)da f'(y)dy ds - \displaystyle \int_{0}^{t}\<\mu_{s},f\>\displaystyle \int_{0}^{1}h_{s}(a)dads\notag\\
&=&\<\mu_0,f\>+\int^t_0\<\mu_s,\frac12\De
f\>ds+\int^t_0\int_{\mathbb{R}}
f(u^{-1}_s(a))h_s(a) dads\notag\\
&& \hspace{.3cm} - \displaystyle \int_{0}^{t}\<\mu_{s},f\>\displaystyle \int_{0}^{1}h_{s}(a)dads\notag\\
&=&\<\mu_0,f\>+\int^t_0\<\frac12\De^{*} \mu_s,
f\>ds+\int^t_0\int_{\mathbb{R}}
f(y)h_s(u_s(y))du_s(y) ds\notag\\
&& \hspace{.3cm} - \displaystyle
\int_{0}^{t}\<\mu_{s},f\>\displaystyle \int_{0}^{1}h_{s}(a)dads\\
&=&\<\mu_0,f\>+\int^t_0\<\frac12\De^{*}\mu_s,f\>ds+\int^t_0\<\mu_s,fh_s(u_s)\>ds\notag\\
&& \hspace{.3cm} -\displaystyle
\int_{0}^{t}\<\mu_{s},f\>\displaystyle \int_{0}^{1}h_{s}(a)dads
\end{eqnarray*}
hence, $\mu\in{H}_{\mu_0}$ and
\begin{equation*}
\frac{\left(\dot{\mu_{t}}- \frac{1}{2} \Delta^{*}
\mu_{t}\right)(dy)}{\mu_{t}(dy)} = h_{t}(u_{t}(y)) - \displaystyle
\int_{0}^{1}h_{t}(a)da.
\end{equation*}
If $h$ satisfies (\ref{2.6}) then $\bar{h}_{s}(a)
\equiv h_{s}(a) - \displaystyle \int_{0}^{1}h_{s}(a)da$ also
satisfies the same equation. To minimize $\displaystyle \int_{0}^{1}
\left|h_{s}(a)\right|^2da$, we choose $h$ such that $\displaystyle
\int_{0}^{1}h_{s}(a)da = 0$. Therefore, $\mu \in \tilde{H}_{\mu_{0}}$ and
\begin{equation*}
\frac{\left(\dot{\mu_{t}}- \frac{1}{2} \Delta^{*}
\mu_{t}\right)(dy)}{\mu_{t}(dy)} = h_{t}(u_{t}(y)).
\end{equation*}
Applying the same argument as in Section 5 establishes Theorem 4.

 \end{document}